\documentclass[a4paper,11pt,oneside,reqno]{amsart}
\usepackage[utf8]{inputenc}
\usepackage{amsmath,amsthm,amsfonts,latexsym,amssymb,bm,enumerate}
\usepackage{ae}
\usepackage{cite}
\usepackage{float}
\usepackage{lmodern}
\usepackage[T1]{fontenc}

\usepackage{color}

\usepackage[colorlinks=true]{hyperref}

\definecolor{dark-red}{rgb}{.54,.0,.0}
\definecolor{dark-green}{rgb}{.0,.4,.0}
\definecolor{dark-blue}{rgb}{.04,.04,.4}

\hypersetup{linkcolor=dark-red, urlcolor=dark-blue, citecolor=dark-green}

\newtheorem{Thm}{Theorem}[section]

\newtheorem{rmrk}[Thm]{Remark}
\newtheorem{prpstn}[Thm]{Proposition}

\renewcommand{\theThm}{\arabic{Thm}} 

\newcommand{\xR}{{\mathcal{R}}}

\renewcommand{\phi}{\varphi}
\renewcommand{\to}{\rightarrow}

\newcommand{\cal}{\mathcal}

\newcommand{\nc}{\newcommand}
\nc{\Om}{\Omega}
\nc{\en}{\mbox{\bf e}_1}
\nc{\snd}{S^{\frac{N}{2}}}
\nc{\sndn}{\frac{S^{\frac{N}{2}}}{N}}
\nc{\xe}{\frac{|x|}{\sqrt{\eps}}}
\nc{\sqrte}{\sqrt{\eps}}
\nc{\epsnd}{\eps^{\frac{N}{2}}}
\nc{\epsndd}{\eps^{\frac{N-2}{2}}}
\nc{\rhoey}{\left[\rho\left(\sqrte|y|\right)\right]}
\nc{\rhoeys}{\rho\left(\sqrte|y|\right)}
\nc{\eps}{\varepsilon}
\nc{\eum}{e^{-\sqrt{\lambda_{1}-\eta}\,|t|}}
\nc{\reta}{{\textstyle\sqrt{\lambda_{1}-\eta}\,}}
\nc{\lum}{\lambda_{1}}
\nc{\fum}{{\textstyle \frac{N-l-1}{|t|}}}
\nc{\nb}{\Delta}
\nc{\al}{\alpha}
\nc{\px}{\phi(y)}
\nc{\ed}{e^{-\sqrt{\lambda_{1}-\mu}\,r}}
\nc{\etr}{e^{-\sqrt{\lambda_{1}-\lambda}\,|t|}}
\nc{\etru}{e^{-\sqrt{1+(\lambda_{1}-\lambda)|t|^2}}}
\nc{\eq}{e^{-\sqrt{\lambda_{1}-\mu}\,|t|}}
\nc{\rum}{|t|^{-\frac{l-1}{2}}}
\nc{\rdo}{|t|^{-\frac{l+1}{2}}}
\nc{\rtr}{|t|^{-\frac{l+3}{2}}}
\nc{\ttr}{|t|^{-\frac{l+3}{2}}}
\nc{\fdo}{{\textstyle\frac{l-1}{2}}}
\nc{\ftr}{{\textstyle\frac{l-3}{2}}}
\nc{\fq}{{\textstyle\frac{l+1}{2}}}
\nc{\rmu}{{\textstyle\sqrt{\lambda_{1}-\mu}\,}}
\nc{\rmul}{{\textstyle\sqrt{\lambda_{1}-\lambda}\,}}
\nc{\hoz}{{\rm H^1_{0}}(\Omega)}
\nc{\g}{g(u)}
\nc{\ux}{u(t,y)}
\nc{\p}{\Psi}
\nc{\un}{\frac{\partial u}{\partial n}}
\nc{\ttt}{|t|^{-\frac{l-1}{2}}}
\nc{\pl}{\Psi'}
\nc{\pll}{\Psi''}
\nc{\wk}{\rightharpoonup}

\begin{document}

\title[Nodal solutions for critical elliptic equations]{SIGN CHANGING SOLUTIONS FOR ELLIPTIC EQUATIONS WITH CRITICAL GROWTH IN CYLINDER TYPE
DOMAINS}\thanks{The authors are partially supported by FCT}

\author{Pedro Gir\~ao}\address{Mathematics Department, Instituto Superior T\'{e}cnico,
Av. Rovisco Pais, 1049-001 Lisboa, Portugal}
\email{pgirao@math.ist.utl.pt}
\author{Miguel Ramos}\address{CMAF and Faculty of Sciences, Universidade de Lisboa, 
Av.\ Prof.\ Gama Pinto, 2, 1649-003 Lisboa, Portugal}
\email{mramos@lmc.fc.ul.pt}

\begin{abstract}
We prove the existence of positive and of nodal solutions for
$-\Delta u  = |u|^{p-2}u+\mu |u|^{q-2}u$, $u\in {\rm H_0^1}(\Omega)$, where $\mu >0$ and $2< q <
p=2N(N-2)$, for a class of open subsets $\Omega$ of $\xR^N$ lying between two infinite cylinders. 
\end{abstract}



\maketitle

\section*{Introduction}

We are concerned with the existence of nonzero solutions for the nonlinear second order elliptic equation
$$
-\Delta u  = |u|^{p-2}u+\mu |u|^{q-2}u, \qquad u\in {\rm H_0^1}(\Omega),\eqno{{\rm (P)}}$$
where $\Omega$ is a smooth unbounded domain of $\xR^N$ with $N\geqslant 3$, $\mu \in \xR^{+}$, $ 2 < q < p$
and $p$ is the critical Sobolev exponent $p=2^*=2N/(N-2)$. Without loss of generality we assume that $0\in
\Omega$.

In the case where $\Omega$ is bounded, the proof of the existence of positive and of nodal (sign changing)
solutions for (P) or similar equations goes back to the work in \cite{BN,CSS,Ta}. In the case where $\Omega$
is unbounded and $p$ is subcritical ($p<2^*$), we refer for example to \cite{DF,Zhu}. On the other hand,
motivated by the work in \cite{BTW,BCS,DF,Li}, in \cite{RWW} the authors prove the existence of a positive
solution for a class of unbounded domains, concerning the (somewhat simpler) equation $-\Delta u  =\lambda u
+|u|^{p-2}u$, where $\lambda$ is positive and small (see also \cite{ST} for a related result).

The present work complements the quoted results. Following \cite{DF,RWW}, we fix a number
$1\leqslant \ell \leqslant N-1$ and write $\xR^N=\xR^{\ell}\times
\xR^{N-\ell}$, $z=(t,y) \in \xR^{\ell}\times
\xR^{N-\ell}$. For a given subset $A \subset \xR^{N-\ell}$ we denote
$A_{\delta}=\{y\in
\xR^{N-\ell}:\mbox{ dist}(y,A)<\delta\}$ and $\widehat{A}=\xR^{\ell}\times A$. Also, for
$t\in \xR^{\ell}$ we let $\Omega^{t}=\{y\in
\xR^{N-\ell}:\,(t,y)\in \Omega\}$. We shall consider both situations (H) and (H)$_0$ below:

\begin{enumerate}
\item[(H)] There exist two nonempty bounded open sets $F\subset G \subset
\xR^{N-\ell}$ such that $F$ is a Lipschitz domain and
$\widehat{F}
\subset \Omega \subset \widehat{G}$.  Moreover,
for each $\delta >0$ there is $R>0$ such that $\Omega^t\subset
F_{\delta}$
for all $|t|\geqslant R$.
\end{enumerate} 

\begin{enumerate}
\item[(H)$_0$] There exists an  open bounded set $G \subset
\xR^{N-\ell}$   such that $\Omega \subset \widehat{G}$ and moreover  for each $\delta >0$ there is $R>0$
such that
$\Omega^t\subset 
{B}_{\xR^{N-\ell}}(0,\delta)$  for all $|t|\geqslant R$.
\end{enumerate}
We have denoted by $B_{\xR^{N-\ell}}(0,\delta)$ the open
ball in $\xR^{N-\ell}$ centered at the origin with radius $\delta >0$. The case (H)$_0$ can be seen as a
limit case of (H), with $F=\{0\}$. We prove the following.

\begin{Thm}
Consider problem {\rm (P)} with $2 < q < p = 2^*$ and assume  either {\rm (H)} or {\rm
(H)$_0$}. Then, for every $\mu>0$, the problem admits a positive (and a
negative) solution of least energy.
\end{Thm}

In order to prove the existence of nodal solutions in case (H), we impose further restrictions on
$\Omega$, namely that $\Omega$ approaches $\widehat{F}$ ``smoothly and slowly." 

\begin{enumerate}
\item[(H)$'$]  Assume (H) and that $\Omega$ is of class ${\rm C^{1,1}}$ in such a way that the local charts as
well as their inverses have uniformly bounded Lipschitz constants. Moreover, there exist constants $m>0$ and
$0<a_1<a_0$ such that 
$\left(1+\frac{a}{|t|^m}\right)\;F \subset \Omega^{t}$ for every $a\in [a_1,a_0]$ and every $|t|$
large.
\end{enumerate}

\begin{Thm}
Consider problem {\rm (P)} with $2 < q < p = 2^*$ and assume  either {\rm (H)$'$} or
{\rm (H)$_0$}. In case {\rm (H)$_0$} holds, assume moreover that
$q>(N+2)/(N-2)$. Then, for every $\mu>0$, the problem admits a  sign changing solution.
\end{Thm}

In Theorem 2 the conclusion is that (P)  has a pair of sign changing solutions, since
the nonlinearity is odd. In case {\rm (H)$_0$}, the extra restriction on $q$ is merely needed in lower
dimensions ($N=3,4,5$), since
$(N+2)/(N-2)\geqslant 2$ for $ N\geqslant 6$. In fact, Theorem 2 still holds if $q=(N+2)/(N-2)$ provided
$\mu$ is sufficiently large (see the remark which follows the proof of Proposition \ref{c1-estimate-(H)0}).

The proof of our main theorems is given in Section 2 (see
Propositions \ref{c0-is-attained} and \ref{c1-is-attained}); it relies on the concentration-compactness
principle at infinity and on some ideas of
\cite{CSS,RWW}. Section 3 provides  technical estimates which are needed in the proof of Theorem 2. We
also give further information on the decay properties of the solutions found in Theorems 1 and 2.

\setcounter{equation}{0}
\renewcommand{\theThm}{\thesection.\arabic{Thm}}

\section{Concentration-compactness}

It is well known that the solutions of (P) correspond to critical points of the energy
functional (for simplicity of notations, we take $\mu=1$ in (P)):
$$
I(u)=\frac{1}{2}\int |\nabla u|^2 - \frac{1}{p}\int |u|^p - \frac{1}{q}\int |u|^q,\qquad u \in
{\rm H_0^1}(\Omega),$$
where the integrals are taken over the domain $\Omega$. We recall $2 < q < p = 2^*$. It follows
from assumptions (H) or (H)$_0$ that we can choose the norm
$||u||:=\left(\int |\nabla u|^2\right)^{1/2}$  in ${\rm H_0^1}(\Omega)$. Let
\begin{equation}\label{c0}
c_0:=\inf\{ I(u): \; u\in {\rm H_0^1}(\Omega),\; u\neq 0 \; {\rm and} \; I'(u)u=0\}. 
\end{equation}
It is also clear that $c_0 >0$ and that every nonzero critical point $u$ of $I$ is such that $I(u)\geqslant
c_0$. The following result proves  Theorem 1.

\begin{prpstn}\label{c0-is-attained}  Under assumptions {\rm (H)} or {\rm (H)$_0$}, the infimum in 
{\rm (\ref{c0})} is attained in a critical point of $I$.\end{prpstn}
\begin{proof} 1. We shall omit  what concerns  standard arguments (cf.\ \cite{BN,CSS}). We first
recall that there exists a Palais-Smale sequence $(u_n)\subset {\rm H^1_0}(\Omega)$ at level $c_0$,
namely
\begin{equation}\label{PS-sequence}
I(u_n)\to c_0 \qquad \mbox{ and } \qquad  I'(u_n)\to 0.
\end{equation}
Since moreover $c_0 >0$, (\ref{PS-sequence}) implies that $\liminf ||u_n||>0$. This sequence is bounded and,
up to a subsequence, $u_n\rightharpoonup u$ weakly in ${\rm H^1_0}(\Omega)$, $u_n(x)\to u(x)$ a.e.\ and $I'(u)=0$,
$I(u)\geqslant 0$. Since  $ \liminf ||u_n||>0 $ and  $I'(u_n)u_n\to 0$, we also have that $\liminf
\int|u_n|^p >0$; indeed, if $\int|u_n|^p \to 0$ along a subsequence, then, since $(\int u_n^2)$ is
bounded, by interpolation  $\int|u_n|^q \to 0$, whence $||u_n||\to 0$, as
$I'(u_n)u_n\to 0$.\newline 
2. 
Up to subsequences, there exist measures $\mu$ and
$\nu$ on
$\Omega$ such that
$|\nabla(u_n-u)|^2
\rightharpoonup \mu$ and $|u_n-u|^p\rightharpoonup \nu$ weakly in the space $M(\Omega)$ of finite measures
in $\Om$. Clearly, $||\mu||\geqslant S ||\nu||^{2/p}$, where $S$ is the  best constant for the embedding
${\rm H^1}(\xR^N)\subset {\rm L^p}(\xR^N)$. By testing $I'(u_n)\to 0$ with $u_n\varphi$ for any $\varphi \in
{\cal D}(\xR^N)$ and since $I'(u)u\varphi=0$ we also see that
\begin{equation}\label{equal-local-mass}
||\mu||=||\nu||. 
\end{equation}
In particular,
\begin{equation}\label{nonzero-local-mass}
\mu\neq 0 \Rightarrow ||\mu||\geqslant S^{p/(p-2)}=S^{N/2}. 
\end{equation}
3. Define
\begin{eqnarray*}
\displaystyle \mu_{\infty} 
&:=& \lim_{R\to \infty}\limsup_{n\to \infty} \int_{|x|>R}|\nabla u_n|^2,\\
\displaystyle \nu_{\infty}
&:=& \lim_{R\to \infty}\limsup_{n\to \infty} \int_{|x|>R}|u_n|^{p},\\
\displaystyle \eta_{\infty}
&:=& \lim_{R\to \infty}\limsup_{n\to \infty} \int_{|x|>R}|u_n|^{q}.
\end{eqnarray*}
Again, it is clear that
\begin{equation}\label{nonzero-mass-at-infinity}
\mu_{\infty} \geqslant S\;\nu_{\infty}^{2/p}.
\end{equation}
By testing $I'(u_n)\to 0$ with $u_n\psi_R$ ($R>0$) where $\psi_R\in {\rm C}^{\infty}(\Omega)$, $0\leqslant
\psi_R\leqslant 1$ is such that $\psi_R(x)=0$ if $|x|\leqslant R$ and $\psi_R(x)=1$ if $|x|\geqslant R+1$,
it follows easily that
\begin{equation}\label{equal-mass-at-infinity}
\mu_{\infty}=\nu_{\infty}+\eta_{\infty}. 
\end{equation} 
4. We recall from \cite{BTW,BCS,Wi} that
\begin{eqnarray*}
\displaystyle \int |\nabla u_n|^2 
&=& \int |\nabla u|^2 + ||\mu|| + \mu_{\infty}+{\rm o}(1),\\
\displaystyle \int | u_n|^p 
&=& \int | u|^p + ||\nu|| + \nu_{\infty}+{\rm o}(1),\\
\displaystyle \int | u_n|^q 
&=& \int | u|^q + \eta_{\infty}+{\rm o}(1).
\end{eqnarray*}
As a consequence, and thanks to (\ref{PS-sequence}), (\ref{equal-local-mass}) and
(\ref{equal-mass-at-infinity}), we have that
\begin{equation}\label{limit-equation}
c_0=I(u)+\left(\frac{1}{2}-\frac{1}{p}\right)||\mu||+\left(\frac{1}{2}-\frac{1}{p}\right)\nu_{\infty}
+\left(\frac{1}{2}-\frac{1}{q}\right)\eta_{\infty}.
\end{equation}
In particular, $c_0\geqslant I(u)$. Since $I'(u)=0$, the proof will be complete once we show that $u\neq 0$.
Indeed, in this case we have that $I(u)\geqslant c_0$, whence $I(u)=c_0$.  (Incidentally,
(\ref{equal-mass-at-infinity}) and (\ref{limit-equation}) also show that, in fact, $||\mu||=\mu_{\infty}=0$,
hence $u_n\to u$ in ${\rm H^1_0}(\Omega)$.) \newline
5. We recall from \cite{BN} that $c_0 < S^{N/2}/N$. Since (\ref{limit-equation}) implies that
$$
c_0\geqslant \left(\frac{1}{2}-\frac{1}{p}\right)||\mu||=\frac{1}{N}\;||\mu||,$$
we deduce from (\ref{equal-local-mass})-(\ref{nonzero-local-mass}) that $\mu=\nu=0$. Thus
$u_n\to u$ in ${\rm H}^1_{{\rm loc}}(\Omega)$ and
\begin{equation}\label{new-limit-equation}
c_0=I(u)+\left(\frac{1}{2}-\frac{1}{p}\right)\nu_{\infty}
+\left(\frac{1}{2}-\frac{1}{q}\right)\eta_{\infty}.
\end{equation}
6.  Suppose first that  $\Omega=\widehat{F}$.
Since
$\liminf\int|u_n|^p >0$, by Lemma 2.1 in \cite{RWW}  we may assume that, up to translations,
$\int_{B_1(0)}|u_n|^p\geqslant c$ for some $c>0$. Since $u_n\to u$ in
${\rm H}^1_{{\rm loc}}(\Omega)$, we conclude that $u\neq 0$ and this proves Proposition \ref{c0-is-attained} for
the case
$\Omega=\widehat{F}$. Moreover, the argument shows that
$c_0(\widehat{F_{\delta}})\to c_0(\widehat{F})$ as $\delta \to 0$ (see (H) and (\ref{c0(widehat{F})}) for
the notations).\newline
7.  We complete the proof  in case (H)$_0$
holds. Assume by contradiction that $u=0$. Then, clearly $\int u_n^2\to 0$ (see
e.g.\ (2.1) in \cite{RWW}). By interpolation, also
$\int |u_n|^q\to 0$. In particular, $\eta_{\infty}=0$. Since $c_0<S^{N/2}/N$,
(\ref{nonzero-mass-at-infinity}), (\ref{equal-mass-at-infinity}) and (\ref{new-limit-equation}) show that
then $\nu_{\infty}=0$, whence, by the second identity in Step~4,
$\int |u_n|^p \to 0$. This contradicts the fact that $\liminf \int |u_n|^p >0$ and proves Proposition
\ref{c0-is-attained} under (H)$_0$.\newline
8. At last, we consider the case where (H) holds and $\Omega \neq \widehat{F}$. Again, assume by
contradiction that
$u=0$. Let
$\delta >0$ be given and take $R>0$ according to assumption (H). Let $\psi_R$ be as in Step~3 and denote
$$
v_n=u_n\psi_R\in {\rm H^1_0}(\widehat{F_{\delta}}).$$
Since $u_n \to 0$ in ${\rm H}^1_{{\rm loc}}(\Omega)$, clearly we have that
\begin{equation}\label{truncated-PS-sequence}
I(v_n)=I(u_n)+{\rm o}(1)\qquad\mbox{ and } \qquad I'(v_n)v_n={\rm o}(1).
\end{equation}
We claim that 
\begin{equation}\label{nontrivial-truncated-PS-sequence}
I(v_n)+{\rm o}(1)\geqslant c_0(\widehat{F_{\delta}}).
\end{equation}
Assuming the claim for a moment, it follows from
(\ref{truncated-PS-sequence})-(\ref{nontrivial-truncated-PS-sequence}) that
$$
c_0=I(u_n)+{\rm o}(1)=I(v_n)+{\rm o}(1)\geqslant c_0(\widehat{F_{\delta}}).$$
Since $\delta >0$ is arbitrary, we conclude that $c_0\geqslant c_0(\widehat{F})$. On the other hand, since
$\widehat{F}\subset \Omega$ and $c_0(\widehat{F})$ is attained (see Step~6 above), we must have
that $c_0 < c_0(\widehat{F})$. This contradiction completes the proof.

It remains to prove the inequality in (\ref{nontrivial-truncated-PS-sequence}). For this, we observe that
(\ref{truncated-PS-sequence}) together with the fact that $\liminf I(u_n) >0$ implies that $\liminf
||v_n||>0$ and $\liminf \int |v_n|^p >0$. Now, let
$$
w_n=t_nv_n\qquad (t_n>0)$$
be such that $I'(w_n)w_n=0$; namely, $t_n$ is given by
$$
\frac{t_n^{p-2}\int|v_n|^p + t_n^{q-2}\int|v_n|^q}{\int |\nabla v_n|^2}=1.$$
Then $(t_n)$ is bounded and,  since $I'(v_n)v_n\to 0$, we see that $t_n \to 1$. In particular,
\begin{equation}\label{same-level}
I(w_n)=I(v_n)+{\rm o}(1).
\end{equation}
Now, by definition, $I(w_n)\geqslant c_0(\widehat{F_{\delta}})$ and
(\ref{nontrivial-truncated-PS-sequence}) follows from (\ref{same-level}).\end{proof}

Using the notation in assumption (H), we denote
\begin{equation}\label{c0(widehat{F})}
c_0(\widehat{F}):=\inf\{ I(u): \; u\in {\rm H_0^1}(\widehat{F}),\; u\neq 0 \; {\rm and} \; I'(u)u=0\}<S^{N/2} /N. 
\end{equation}
We also let
\begin{equation}\label{c0-infty}
c_0^{\infty}:=c_0(\widehat{F}) \quad\mbox{in case (H)}, \qquad c_0^{\infty}:=S^{N/2} /N\quad\mbox{ in case
(H)}_0. 
\end{equation}
 We have shown in the proof of Proposition \ref{c0-is-attained} that $c_0(\widehat{F})$ is attained by a
critical point of the energy functional in ${\rm H_0^1}(\widehat{F})$. In fact, the argument above  yields the
following compactness result.

\begin{prpstn}\label{compactness-lemma} Under assumptions {\rm (H)} or {\rm (H)$_0$}, let $(u_n)\subset 
{\rm H_0^1}(\Omega)$ be such that
\begin{eqnarray}\label{compactness}
 \limsup I(u_n) < c_0^{\infty}\qquad  {\rm and} \qquad  I'(u_n)(u_n\psi)\to 0
\end{eqnarray}
for every $\psi \in {\rm C}^{\infty}(\Omega)\cap W^{1,\infty}(\Omega)$. Suppose $u_n\rightharpoonup u$ weakly in
${\rm H^1_0}(\Omega)$, $u_n(x)\to u(x)$ a.e.\ and $I'(u)(u\psi)=0$ for such functions $\psi$. Then $u_n\to u$ in
${\rm H^1_0}(\Omega)$.\end{prpstn} 

\begin{proof}  Since $I'(u)u=0$, we have that $I(u)\geqslant 0$. Denote $v_n:=u_n-u$. By the Brezis-Lieb Lemma,
$$
I(v_n)=I(u_n)-I(u)+{\rm o}(1) <  c_0^{\infty}+{\rm o}(1) $$
and
$$
I'(v_n)(v_n\psi)=I'(u_n)(u_n\psi)-I'(u)(u\psi)+{\rm o}(1)\to 0$$
for every $\psi \in {\rm C}^{\infty}(\Omega)\cap  W^{1,\infty}(\Omega)$. Since $(v_n)$ converges weakly to zero,
a similar (though easier) argument as in the proof of Proposition \ref{c0-is-attained} shows that we cannot
have $\limsup I(v_n) >0$. Thus 
$
I(v_n)\to 0.$
Since also $I'(v_n)v_n \to 0$, we conclude that $||v_n||\to 0$, hence $u_n\to u$ in ${\rm H_0^1}(\Omega)$.\end{proof}

Next we turn to the proof of Theorem 2. Following \cite{CSS}, let
\begin{equation}\label{c1}
c_1:=\inf\{ I(u): \; u\in {\rm H_0^1}(\Omega),\; u^{\pm}\neq 0 \; {\rm and} \; I'(u^{\pm})u^{\pm}=0\}\geqslant
c_0 >0, 
\end{equation}
where we denote $u^{+}=\max\{u,0\}$ and $u^{-}=\max\{-u,0\}$. The following proposition will be proved in
Section 3 (cf.\ Propositions \ref{c1-estimate-(H)'} and \ref{c1-estimate-(H)0}).

\begin{prpstn}\label{c1-estimate} Assume {\rm (H)$'$} or {\rm (H)$_0$} holds; in the latter case, we also
assume that  
$q>(N+2)/(N-2)$. Then
$$
c_1 < c_0 + c_0^{\infty}.$$ \end{prpstn}
 
Our final result completes the proof of Theorem 2.

\begin{prpstn}\label{c1-is-attained} Assume {\rm (H)$'$} or {\rm (H)$_0$} holds; in the latter case, we also
assume that  
$q>(N+2)/(N-2)$.  Then the infimum in {\rm (\ref{c1})} is
attained in a critical point of $I$.\end{prpstn} 

\begin{proof} It is known (cf.\ \cite{CSS}) that there exists a Palais-Smale sequence at level $c_1$, namely
\begin{eqnarray*}
I(u_n)\to c_1 \qquad \mbox{ and } \qquad  I'(u_n)\to 0,
\end{eqnarray*}
with the additional property that
\begin{equation}\label{nontrivial.pos.part}
I(u_n^{\pm})\geqslant c_0 +{\rm o}(1)
\end{equation}
(so that, in fact, $c_1 \geqslant 2c_0$). As in Step~1 in the proof of Proposition \ref{c0-is-attained},
modulo a subsequence, $(u_n)$ converges weakly in ${\rm H_0^1}(\Omega)$ and pointwise a.e.\ to a critical point
$u$ of $I$.  Observe that
$I'(u_n)\to 0$ implies that
\begin{equation}\label{PS.for.pos.part}
I'(u_n^{\pm})(u_n^{\pm}\psi)=I'(u_n)(u_n^{\pm}\psi)\to 0
\end{equation}
for every $\psi \in {\rm C}^{\infty}(\Omega)\cap  W^{1,\infty}(\Omega)$. Similarly,
$I'(u^{\pm})(u^{\pm}\psi)=0$. Since moreover $I(u_n)=  I(u_n^{+})+I(u_n^{-})=c_1+{\rm o}(1)$, we deduce
from (\ref{nontrivial.pos.part}) and Proposition \ref{c1-estimate} that
\begin{equation}\label{subcritical.pos.part}
\limsup I(u_n^{\pm})<c_0^{\infty}.
\end{equation}
It follows from (\ref{PS.for.pos.part}), (\ref{subcritical.pos.part}) and Proposition
\ref{compactness-lemma} that
$u_n^{\pm}\to u^{\pm}$ in ${\rm H_0^1}(\Omega)$. Hence $u_n \to u$ in ${\rm H_0^1}(\Omega)$, $I(u)=c_1$ and 
$I(u^{\pm})\geqslant c_0>0$. This finishes the proof.\end{proof}

\setcounter{equation}{0}

\section{Decay and energy estimates}

This section is devoted to general equations of the form
\begin{equation}\label{general-equation}
-\Delta u -\lambda u = g(u), \qquad u\in {\rm H^1_0}(\Omega),
\end{equation}
where $\Omega\subset \xR^N$ ($N\geqslant 3$) is an open set with ${\rm C}^{1,1}$ boundary and $g$ satisfies
(recall that $ p=2^*=2N/(N-2)$)
\begin{equation}\label{growth-on-g}
|g(s)|\leqslant C\;(|s|+|s|^{p-1}),\qquad \forall s\in \xR.
\end{equation}
Under assumption (\ref{growth-on-g}), it follows from the Brezis-Kato estimates and classical elliptic
regularity theory that the solutions of (\ref{general-equation}) lie in
${\rm C^2}(\Omega)\cap {\rm L^{\infty}}(\Omega)\cap {\rm C}(\overline{\Omega})$. In view of the applications that we have in
mind (cf.\ assumptions (H)-(H)$_0$), we let $\xR^N=\xR^{\ell}\times \xR^{N-\ell}$
with
$1\leqslant \ell < N$ and accordingly write $(t,y)\in \xR^{\ell}\times \xR^{N-\ell}$ for any point $(t,y)\in
\xR^N$.

\begin{prpstn}\label{decay} Let $\Om=\xR^{\ell}\times F$ where $F\subset\xR^{N-\ell}$ is a ${\rm C^{1,1}}$ domain
and let 
$g\in {\rm C^1}(\xR)$ satisfy {\rm (\ref{growth-on-g})}, $g(0)=0$ and $g'(s)=\mbox{o}(s^\eps)$ near 0, for some
$\eps>0$. Let $u$ be a solution of
\begin{equation}\label{general-equation-bis}
-\Delta u -\lambda u= g(u), \qquad u\in {\rm H^1_0}(\Omega),
\end{equation}
where $\lambda<\lambda_1$  and $\lambda_1$ is the first eigenvalue 
of $(-\Delta,{\rm H^1_0}(F))$. Then
\begin{equation}\label{eigendecay}
|u(t,y)|+|\nabla_tu(t,y)|\leqslant \phi(y)e^{-\sqrt{1+(\lambda_1-\lambda)|t|^2}},\qquad \forall(t,y)\in\Om,
\end{equation}
where  $\phi$ is a positive eigenfunction associated to $\lambda_1$.  Also, there exists a
constant $C>0$ such that
$$
|\nabla u(t,y)|\leqslant  C e^{-\sqrt{1+(\lambda_1-\lambda)|t|^2}},\qquad \forall(t,y)\in\Om.
$$
\end{prpstn}

\begin{proof}  1. Since $u\in {\rm L^{\infty}}(\Omega)$, we have from (\ref{growth-on-g}) that $|g(u(x))|\leqslant
c|u(x)|$ for every
$x\in
\Omega$. By elliptic regularity theory (Theorem 9.15 of \cite{GT}),
there exists $c>0$ such that, for all $\alpha \geqslant 2$, 
$$
||u||_{W^{2,\alpha}(B_1(0)\times F)}\leqslant c\;||u||_{{\rm L}^{\alpha}(B_2(0)\times F)}.
$$
Due to invariance by translations,
\begin{equation}\label{decum}
||u||_{W^{2,\alpha}(B_1(t)\times F)}\leqslant c\;||u||_{{\rm L}^{\alpha}(B_2(t)\times F)}\qquad\forall
t\in\xR^{\ell}.
\end{equation}
In particular, 
\begin{equation}\label{decdois}
u(t,y)\to 0\mbox{\ as\ }|t|\to+\infty,\mbox{\ uniformly for\ }y\in F
\end{equation}
and
\begin{equation}\label{decdoist}
|\nabla u(t,y)|\to 0\mbox{\ as\ }|t|\to+\infty,\mbox{\ uniformly for\ }y\in F.
\end{equation}
2. Suppose $\mu\in]\lambda,\lambda_1[$ is fixed and let
$$
\Psi(t):=\alpha e^{-\sqrt{1+(\lambda_1-\mu)|t|^2}}\in {\rm H^1}(\xR^{\ell}),
$$
where $\alpha$ will be chosen later.  An easy computation shows that
\begin{equation}\label{main-eigencomputation}
-\Delta\Psi+(\lambda_1-\mu)\Psi=(\lambda_1-\mu)\,\Psi\,((\ell -1)\theta^{-1/2}+\theta^{-1}+\theta^{-3/2})
\end{equation}
where $\theta(t):=1+(\lambda_1-\mu)|t|^2$. In particular,
$$
-\Delta\Psi+(\lambda_1-\mu)\Psi\geqslant\frac{\alpha(\lambda_1-\mu)}{1+(\lambda_1-\mu)|t|^2}
e^{-\sqrt{1+(\lambda_1-\mu)|t|^2}}=:h(t).
$$
Let $\phi$ be a positive eigenfunction associated to  $\lambda_1$
 and
$$
z(t,y):=\phi(y)\Psi(t).
$$
The function $z$ satisfies
$$
-\Delta z-\mu z\geqslant\phi(y)h(t).
$$
Hence, for $w:=z-u$, we have
\begin{equation}\label{upper-solution}
-\Delta w-\mu w\geqslant \phi(y)h(t)+(\mu-\lambda)u-g(u)=:k(t,y).
\end{equation}
Since $g(0)=0=g'(0)$, it follows from (\ref{decdois}) that if $u(t,y)\geqslant 0$, then
$$
(\mu-\lambda)u-g(u)\geqslant 0
$$
if $|t|>R$, where $R$ is chosen large; hence also  $k(t,y)\geqslant 0$.  In summary, 
\begin{equation}\label{maximum-principle}
w<0\Rightarrow-\Delta w-\mu w\geqslant 0,
\end{equation}
if $|t|>R$. Since $\partial z/\partial \nu =h\;\partial \varphi/\partial \nu<0$ ($\nu$ stands for the
outward normal to $\partial \Omega$), we can fix $\alpha$ so large that $w\geqslant 0$ for $|t|\leqslant R$.
Let $\omega:=\{x\in \Omega:
w(x)<0\}$. Since
$$
w^{-}(x)=0\qquad \forall x\in \partial \omega,
$$
by multiplying (\ref{upper-solution}) by $w^{-}$ and integrating, it follows from (\ref{maximum-principle})
that $\omega=\emptyset$. Therefore $u\leqslant z$.  In the same way we can prove that
$-u\leqslant z$, and so
\begin{equation}\label{decqua}
|u(t,y)|\leqslant\phi(y)e^{-\sqrt{1+(\lambda_1-\mu)|t|^2}},\qquad \forall(t,y)\in\Om;
\end{equation}
the constant $\alpha$ has been incorporated into the function $\phi$.\newline
3.  We now improve the previous estimate.  Since $g'(s)=o(s^\eps)$, there exists $C>0$
such that
\begin{equation}\label{deccin}
|g(u(t,y))|\leqslant C|u(t,y)|^{1+\eps},\qquad\forall(t,y)\in\Om.
\end{equation}
We fix $\mu\in]\lambda,\lambda_1[$, sufficiently close to $\lambda$, so that
$$
\gamma:=(1+\eps)\sqrt{\lambda_1-\mu}>\sqrt{\lambda_1-\lambda}.
$$
Combining (\ref{decqua}) and (\ref{deccin}), 
\begin{equation}\label{decseis}
|g(u(t,y))|\leqslant C\phi(y)^{1+\eps}e^{-\gamma|t|},\qquad\forall(t,y)\in\Om.
\end{equation}
Let $z(t,y):=\phi(y)\Psi(t)$, where $\Psi$ is like in Step~2, with $\mu$
replaced by $\lambda$.  For $w:=z-u$, we have
$$
-\Delta w-\lambda w\geqslant
\frac{\alpha(\lambda_1-\lambda)}{1+(\lambda_1-\lambda)|t|^2}\phi(y)
e^{-\sqrt{1+(\lambda_1-\lambda)|t|^2}}
-g(u(t,y))=:p(t,y).
$$
Since $\gamma>\sqrt{\lambda_1-\lambda}$, it follows from (\ref{decseis}) that $p(t,y)\geqslant 0$ if
$|t|$ is large. Choosing $\alpha$ sufficiently large leads to $p\geqslant 0$ in $\Om$.  We conclude 
from the maximum principle, as before, that $u\leqslant z$ in $\Om$ and in the same way, $|u|\leqslant z$ in
$\Om$. \newline
4. To finish the proof we use the decay of $u$.  Specifically, the derivatives
$v=\partial u/\partial t_i$, for $i=1,\ldots,\ell$, satisfy
$$
-\Delta v-\lambda v=g'(u)v\qquad\mbox{and\ }\; v\in {\rm H^1_0}(\Om).
$$
The argument in Steps~2 and~3 above proves an analogous decay for $v$. The main point in the final argument
is that if $\mu\in]\lambda,\lambda_1[$ is sufficiently close to $\lambda$ then
$$
\frac{\alpha(\lambda_1-\lambda)}{1+(\lambda_1-\lambda)|t|^2}\varphi(y)
e^{-\sqrt{1+(\lambda_1-\lambda)|t|^2}}
-C\varphi^{\varepsilon}(y)e^{-\varepsilon \sqrt{1+(\lambda_1-\lambda)|t|^2}}
\times \varphi(y) e^{-\sqrt{1+(\lambda_1-\mu)|t|^2}}$$
is positive for $|t|$ large. The final assertion in the statement of Proposition \ref{decay} follows from
(\ref{decum}).\end{proof}

We now consider the setting analyzed in Section 2. Again, we denote by $\lambda_1=\lambda_1(F)$
the first eigenvalue  of $(-\Delta,{\rm H^1_0}(F))$.

\begin{prpstn}\label{decay-for-(H)} Suppose $\Om$ is a domain satisfying assumption {\rm (H)} and moreover
that
 $\Omega$ is of class ${\rm C^{1,1}}$ in such a way that the local charts as
well as their inverses have uniformly bounded Lipschitz constants. Let $g\in {\rm C^1}(\xR)$ be as in\/
{\rm Proposition~\ref{decay}} and $u$ be a solution of
$$
-\Delta u-\lambda u=g(u),\qquad u\in {\rm H^1_0}(\Om),
$$
with $\lambda < \lambda_1$.
Then,
for each $\overline{\lambda}\in ]\lambda,\lambda_1[$,
there exists a constant $C>0$ such that
$$
|u(t,y)|+|\nabla u(t,y)|\leqslant C e^{-\sqrt{1+(\overline{\lambda}-\lambda)|t|^2}},\qquad
\forall(t,y)\in\Om.
$$
\end{prpstn}

\begin{proof} The proof is similar to that of Proposition \ref{decay}, so we just stress the differences. 
Thanks to our assumption on $\Omega$, the constant $c$ in (\ref{decum}) can be taken
uniformly bounded, hence (\ref{decdois}) still holds. Now, fix $\delta >0$ in such a way that $\lambda <
\lambda_1(F_{\delta})<\lambda_1$. Running through the argument in Step~2 of the proof of Proposition
\ref{decay} we see that, similarly to  (\ref{decqua}), 
$$
|u(t,y)|\leqslant\phi(y)e^{-\sqrt{1+(\lambda_1(F_{\delta})-\mu)|t|^2}},\qquad \forall(t,y)\in\Om, \,
|t|\geqslant R,
$$
provided $R>0$ is sufficiently large; here, $\mu \in ]\lambda,\lambda_1(F_{\delta})[$ and 
 $\phi$ is an eigenfunction associated to $\lambda_1(F_{\delta})$. Arguing as in Step~3 of the quoted
proof, the previous estimate for $u$ can be improved to
$$
|u(t,y)|\leqslant\phi(y)e^{-\sqrt{1+(\lambda_1(F_{\delta})-\lambda)|t|^2}},\qquad \forall(t,y)\in\Om, \,
|t|\geqslant R.
$$
This clearly implies that we can choose $C>0$ such that
\begin{equation}\label{first-decay}
|u(t,y)|\leqslant C e^{-\sqrt{1+(\lambda_1(F_{\delta})-\lambda)|t|^2}},\qquad \forall(t,y)\in\Om.
\end{equation}
A similar decay estimate for the derivatives of $u$ follows from (\ref{decum}) and (\ref{first-decay}).
Since $\lambda_1(F_{\delta})$ can be chosen arbitrarily close to $\lambda_1$
(see Lemma 2.3 of~\cite{RWW}), this proves the
proposition.\end{proof}

Going back to Proposition \ref{decay}, it may be interesting to observe that the asymptotic estimates can
be sharpened as follows.

\begin{prpstn}\label{optimal-decay} Under the assumptions of\/ {\rm Proposition~\ref{decay}}, let $u$ be a
solution of problem {\rm (\ref{general-equation-bis})}. Then:\\
(a) The conclusion of\/  {\rm Proposition~\ref{decay}} still holds with 
$e^{-\sqrt{1+(\lambda_1-\lambda)|t|^2}}$ replaced by
$e^{-\sqrt{1+(\lambda_1-\lambda)|t|^2}}\;|t|^{-\frac{\ell-1}{2}}$.\\ 
(b)  {\rm (Hopf lemma)} If $u$ is
positive and $\eta < \lambda$ then $u(t,y) \geqslant
\widetilde{\varphi}(y) e^{-\sqrt{1+(\lambda_1-\eta)|t|^2}}$ for every $(t,y)\in \Om$, for some
positive eigenfunction $\widetilde{\varphi}$  associated to $\lambda_1$.
\end{prpstn} 

\begin{proof} (a) We improve the estimate (\ref{eigendecay}) by repeating the argument with
$$
\Psi(t):=e^{-\sqrt{1+(\lambda_1-\lambda)|t|^2}}\;|t|^{-\frac{\ell-1}{2}}.$$
Indeed, 
$$
-\Delta \Psi + (\lambda_1-\lambda)\Psi=\Psi\;\left( (\lambda_1-\lambda)\theta^{-1}+
(\lambda_1-\lambda)\theta^{-3/2}+\frac{\ell - 1 }{2}\, \frac{\ell -3}{2}\, \frac{1}{|t|^2}\right),$$
a computation that can be easily checked using  (\ref{main-eigencomputation}); here, of course, 
$\theta(t):=1+(\lambda_1-\lambda)|t|^2$. As a consequence, for sufficiently large $|t|$ we have that 
$$
-\Delta \Psi + (\lambda_1-\lambda)\Psi\geqslant \frac{1}{2} \,
e^{-\sqrt{1+(\lambda_1-\lambda)|t|^2}}\;|t|^{-\frac{\ell+3}{2}}=:h(t).$$
Due to the assumptions on $g$, for the function on $w:=\alpha \varphi \Psi - u$, with $\alpha$ a fixed
positive number, we have
$$
-\Delta w - \lambda w \geqslant \alpha h(t)\varphi(y) - A
\varphi^{1+\varepsilon}(y)e^{-(1+\varepsilon)\sqrt{1+(\lambda_1-\lambda)|t|^2}}.$$
The right hand member above is positive for sufficiently large $|t|$. Using the maximum principle, we
conclude, as in (\ref{decqua}), that
\begin{equation}\label{optimal-for-u}
|u(t,y)|\leqslant \alpha \varphi(y) e^{-\sqrt{1+(\lambda_1-\lambda)|t|^2}}\;|t|^{-\frac{\ell-1}{2}},\qquad
\forall(t,y)\in\Om.
\end{equation}
Finally, as in Step~4 of the quoted proof, a similar estimate for the derivatives of $u$ follows from
(\ref{eigendecay}), (\ref{optimal-for-u}) and the fact that
\begin{eqnarray*}
&&\frac{\alpha}{2}\varphi(y)
e^{-\sqrt{1+(\lambda_1-\lambda)|t|^2}}|t|^{-\frac{\ell +3}{2}}-\\
&&\qquad\qquad\qquad-C\varphi^{\varepsilon}(y)e^{-\varepsilon \sqrt{1+(\lambda_1-\lambda)|t|^2}}|t|^{-\varepsilon \frac{\ell
-1}{2}}
\times \varphi(y) e^{-\sqrt{1+(\lambda_1-\lambda)|t|^2}}
\end{eqnarray*}
is positive for $|t|$ large.\\
(b) Here we let $
\Psi(t):=e^{-\sqrt{1+(\lambda_1-\eta)|t|^2}}$. Fix any $\mu\in ]\eta,\lambda[$. Similarly to
(\ref{main-eigencomputation}), we can check that
$$
h(t):=-\Delta \Psi + (\lambda_1-\mu)\Psi \leqslant 0 \quad \mbox{ for every } |t|\geqslant R$$
with $R$ sufficiently large. Since $u(t,y)\to 0$ as $|t|\to \infty$ and since $g(0)=0=g'(0)$ we can choose
$R$ in such a way that also $(\mu-\lambda)u-g(u)\leqslant 0$ for $|t|\geqslant R$. Letting $z:=\varphi
\Psi$, we can fix a small $\alpha>0$ so that $w:=\alpha z - u \leqslant 0$ if $|t|\leqslant R$; this is
possible because $u\in {\rm C^1}(\overline{\Omega})$, $u>0$ in $\Omega$ and $\partial u/\partial \nu < 0$ on
$\partial
\Omega$ (outward normal derivative). In summary, we have that (compare with (\ref{upper-solution}))
$$
-\Delta w - \mu w =\alpha \varphi h + (\mu-\lambda)u-g(u)=:k(t,y)$$
and $k(t,y)\leqslant 0$ for $|t|\geqslant R$, while $w\leqslant 0$ for $|t|\leqslant R$. Using the maximum
principle as in the proof of Proposition \ref{decay} we conclude that $w\leqslant 0$ for all $(t,y)$. \end{proof}

We end this section with the proof of Proposition \ref{c1-estimate}, which is contained in Propositions 
\ref{c1-estimate-(H)'} and \ref{c1-estimate-(H)0} below.
We will refer to the functional
$I$ introduced at the beginning of Section 2 as well as to the quantities $c_0$, $c_0^{\infty}$ and $c_1$
defined in (\ref{c0}), (\ref{c0-infty}) and (\ref{c1}), respectively.\\

\begin{prpstn}\label{c1-estimate-(H)'} Assume {\rm (H)$'$} holds. Then
$
c_1 < c_0 + c_0^{\infty}.$ 
\end{prpstn} 

\begin{proof} 1. We know that $c_0$ is attained by a positive function $v\in {\rm H_0^1}(\Omega)$ and $c_0^{\infty}$ is
attained by some positive function $\psi\in {\rm H_0^1}(\xR^{\ell}\times F)$ (cf.\ Proposition
\ref{c0-is-attained}). Let $m>0$ and $0 < a_1<a_0$ be given by assumption (H)$'$ and denote $A:=a_0/a_1 >1$.
Fix a large number $M$ such that $M>2A$ and
\begin{equation}\label{definition-of-M}
\frac{a_1}{a_0} < \left(\frac{M-A}{M+A}\right)^m.\end{equation}
Let $\rho:\xR\to\xR$ be a smooth function such that $\rho(s)=1$ for $|s|\leqslant 1$ and
$\rho(s)=0$ for $|s|\geqslant A$.
We define
$\rho_R$ and $\eta_R$ in $\xR^{\ell}$ by $\rho_R=\rho(|t|/R)$ and
$\eta_R(t)=\rho_R(t-MR\en)=\rho(|\frac{t}{R}-M\en|)$. We also let
$$v_R(t,y):=v(t,y)\rho_R(t)$$
and
$$\psi_R(t,y):=\lambda_R^{-N/p}\;\psi\left(\frac{t-MR\en}{\lambda_R},\frac{y}{\lambda_R}\right)\eta_R(t),$$
where
\begin{equation}\label{definition-of-lambda-R}
\lambda_R:=1+\frac{a_0}{(M+A)^mR^m}\cdot
\end{equation}
We observe that $v_R$ and $\psi_R$ have disjoint supports. Moreover, both functions  belong to
${\rm H_0^1}(\Omega)$ if $R$ is sufficiently large. Indeed, suppose $(t,y)\in \partial \Omega$ and let us show
that
$\psi_R(t,y)=0$. We may already assume that $|t-MR\en|\leqslant AR$. In particular,
\begin{equation}\label{t-sandwich}
(M-A)R\leqslant |t| \leqslant (M+A)R.\end{equation}
Now, to prove the claim it is sufficient to show that $(\frac{t-MR\en}{\lambda_R},\frac{y}{\lambda_R})
\notin \widehat{F}$, i.e.\ that $\frac{y}{\lambda_R} \notin F$. Observing that
$$
y=\left(1+\frac{a}{|t|^m}\right)\; \frac{y}{\lambda_R}$$
where, according to (\ref{definition-of-M})-(\ref{t-sandwich}),
$$
a:=a_0\; \left(\frac{|t|}{(M+A)R}\right)^m \in [a_1,a_0],$$
the conclusion follows from (H)$'$ and the fact that $(t,y)\notin \Omega$.\newline
2. Thanks to Proposition \ref{decay-for-(H)} (with $\lambda=0$), we know that $|v(t,y)|+|\nabla v(t,y)|={\rm
O}(e^{-\delta|t|})$ and similarly for $\psi$. Here and henceforth $\delta$ denotes
various positive constants. It then follows easily that $I(v_R)\to
I(v)$ and $I(\psi_R)\to I(\psi)$ as $R\to \infty$ and also that
\begin{equation}\label{exponential-estimate}
I(v_R)= I(v)+{\rm O}(e^{-\delta R}), \qquad
I(\psi_R)= I(\psi)+{\rm O}(e^{-\delta R}).
\end{equation}
In fact, the second estimate can be improved, observing that 
$$
\int \psi_R^p=\int\psi^p\rho_R^p=\int \psi^p+\int \psi^p(\rho_R^p-1)=\int \psi^p + {\rm O}(e^{-\delta R})$$
and similarly $\int |\nabla \psi_R|^2=\int |\nabla \psi|^2 + {\rm O}(e^{-\delta R})$, 
while
$$
\int \psi_R^q=\lambda_R^{N(1-\frac{q}{p})}\int\psi^q+{\rm O}(e^{-\delta R})
$$
so that
\begin{eqnarray*}I(\psi_R)&= & I(\psi) + \left(1-\lambda_R^{N(1-\frac{q}{p})}\right) \int\psi^q+ {\rm
O}(e^{-\delta R})\\
&\leqslant & I(\psi) - N\left(1-\frac{q}{p}\right)\; \frac{a_0}{(M+A)^mR^m} \;\int \psi^q + {\rm O}(e^{-\delta
R}),\end{eqnarray*}
whence, for every sufficiently large $R$,
\begin{equation}\label{strict-inequality-for-psi}
I(\psi_R) < I(\psi).\end{equation}
3. Clearly, as in (\ref{exponential-estimate})-(\ref{strict-inequality-for-psi}), for large $R$ and
uniformly for $\tau_1, \tau_2
\in [1/2,2]$, we have that
\begin{eqnarray*}
I(\tau_1v_R-\tau_2\psi_R)&=&I(\tau_1v_R)+I(\tau_2\psi_R) < I(\tau_1 v)+I(\tau_2\psi)\\
&\leqslant&
\sup_{s\geqslant 0} I(sv)+\sup_{s\geqslant 0} I(s\psi) = c_0+c_0^{\infty}.
\end{eqnarray*}
The last equality above is a direct consequence of the definitions of $c_0$ and $c_0^{\infty}$, by standard
arguments (cf.\ \cite{BN,CSS,Wi}). In summary, there exists $R_0$ such that
\begin{equation}\label{final-energy-estimate-(H)'}
\sup_{1/2 \leqslant \tau_1,\tau_2\leqslant 2} I(\tau_1v_R-\tau_2\psi_R)<c_0+c_0^{\infty},\qquad \forall
R\geqslant R_0.
\end{equation}
4. Thanks to (\ref{final-energy-estimate-(H)'}), to complete the proof it remains to show that there exist
$\tau_1, \tau_2\in [1/2,2]$ and $R\geqslant R_0$ such that
$
w:=\tau_1v_R-\tau_2\psi_R$ satisfies  $I'(w^{\pm})w^{\pm}=0$.
Since $v_R$ and $\psi_R$ have disjoint supports, this amounts to prove that there exist
$\tau_1, \tau_2\in [1/2,2]$ and $R\geqslant R_0$ such that
\begin{equation}\label{I'(v,psi)=0}
I'(\tau_1v_R)v_R=0 \qquad \mbox{ and } \qquad I'(\tau_2\psi_R)\psi_R=0.
\end{equation}
Now, we have that $I'(v_R/2)v_R\to I'(v/2)v >0$ and $I'(2v_R)v_R\to I'(2v)v <0$ as $R\to \infty$ and
similarly for $\psi$. Hence (\ref{I'(v,psi)=0}) follows by applying the intermediate value theorem.\end{proof}

\begin{prpstn}\label{c1-estimate-(H)0} Assume {\rm (H)$_0$} holds and moreover that $q>(N+2)/(N-2)$. Then
$
c_1 < c_0 + c_0^{\infty}=c_0+S^{N/2}/N$. 
\end{prpstn} 

\begin{proof} Let $U(x)=c_N/(1+|x|^2)^{(N-2)/2}$ be the Talenti instanton, normalized in such a way that $\int
|U|^p=\int |\nabla U|^2=S^{N/2}$  (i.e.\ $c_N=(N(N-2))^{(N-2)/4}$). Let $U_{\eps}(x)=\eps^{-N/p}U(x/\eps)$
be its rescaling, so that also
$\int |U_{\eps}|^p=\int |\nabla U_{\eps}|^2=S^{N/2}$. The following argument is similar to that in
\cite{Zhu}, except that we cut down the least energy solution and also $U_\eps$ and estimate the error in
doing so, instead of computing the interference between their energies.

Recall that, without loss of generality, we are assuming that $0\in\Om$. By Proposition
\ref{c0-is-attained},  we know that
$c_0$ is achieved by a positive function $v\in {\rm H^1_0}(\Om)\cap {\rm C^1}(\Om)$. 
Let $\rho$, $\eta:\xR\to\xR$ be smooth functions such that $\rho(s)=1$ for $|s|\leqslant 1$,
$\rho(s)=0$ for $|s|\geqslant 2$, $\eta(s)=0$ for $|s|\leqslant 2$ and
$\eta(s)=1$ for $|s|\geqslant 3$.
We define
$\rho_\eps$ and $\eta_\eps:\xR^N\to\xR$ by $\rho_\eps(x)=\rho(|x|/\sqrt{\eps})$
and $\eta_\eps(x)=\eta(|x|/\sqrt{\eps})$.
We also define
$$u_\eps:=U_\eps\;\rho_\eps\qquad \mbox{  and } \qquad v_\eps:=v\;\eta_\eps .$$ 
It is clear that $u_\eps$ and $v_\eps$ have disjoint supports and that they both belong to ${\rm H_0^1}(\Om)$. 
We can estimate
\begin{eqnarray*}
\int |\nabla v_{\eps}|^2 &\leqslant & \int |\nabla v|^2 +2\left( \int_{2\eps^{1/2}\leqslant |x|\leqslant
3\eps^{1/2}}(|\nabla v|^2\eta_{\eps}^2 + v^2|\nabla \eta_{\eps}|^2)\right)\\
&\leqslant &  \int |\nabla v|^2 + {\rm O}(\eps^{N/2}) +{\rm O}(\eps^{(N-2)/2}) \\
& = &\int |\nabla v|^2 + {\rm O}(\eps^{(N-2)/2}), 
\end{eqnarray*}
while
$$
\int  v_{\eps}^p = \int v^p + \int v^p (\eta_\eps^p-1)\geqslant  \int v^p  -  \int_{|x|\leqslant
3\eps^{1/2}} v^p \geqslant  \int v^p + {\rm O}(\eps^{N/2})
$$
and similarly for  $\int v_\eps^q$, so that
\begin{equation}\label{estimate-for-v}
I(v_\eps)\leqslant I(v) + {\rm O}(\eps^{(N-2)/2}).
\end{equation} 
As for $u_\eps$,
\begin{eqnarray*}
\int |\nabla u_{\eps}|^2 &\leqslant & \int |\nabla U_\eps|^2 + 2\left( \int |\nabla U_\eps|^2\rho_{\eps}^2 +
U_\eps^2|\nabla \rho_{\eps}|^2\right)\\ &\leqslant & S^{N/2}   +{\rm
O}(\eps^{(N-2)/2}),  
\end{eqnarray*}
while,  denoting by $c>0$ some constant which is independent of $\eps$,
$$
\int u_\eps^p \geqslant S^{N/2} + {\rm O}(\eps^{N/2}) \quad \mbox{ and } \quad
\int u_\eps^q \geqslant c\;\eps^{N(1-\frac{q}{p})},$$
as can be checked directly, using the explicit expression of $ U_\eps$. In summary,
\begin{equation}\label{estimate-for-U}
I(u_\eps)\leqslant \left(\frac{1}{2}-\frac{1}{p}\right) \;S^{N/2} +  {\rm O}(\eps^{\frac{N-2}{2}})-
c\;\eps^{N(1-\frac{q}{p})}.
\end{equation} 
Combining (\ref{estimate-for-v}) and (\ref{estimate-for-U}) yields
\begin{equation}\label{estimate-for-U-and-v}
I(u_\eps)+I(v_\eps)\leqslant c_0 +\frac{S^{N/2}}{N} +  c_1\; \eps^{\frac{N-2}{2}}-
c_2\;\eps^{N(1-\frac{q}{p})},
\end{equation}
for some  positive constants $c_1$ and $c_2$. In particular,
\begin{equation}\label{final-energy-estimate-(H)0}
I(u_\eps)+I(v_\eps)< c_0 +\frac{S^{N/2}}{N} 
\end{equation}
if $\eps$ is sufficiently small since, by assumption, $\frac{N-2}{2}>N(1-\frac{q}{p})$; indeed, this
condition is equivalent to $q>p-1=(N+2)/(N-2)$. From (\ref{final-energy-estimate-(H)0}) we can end the proof
of Proposition \ref{c1-estimate-(H)0} with similar arguments as in Steps~3 and~4 in the proof of Proposition
\ref{c1-estimate-(H)'}. \end{proof}

\begin{rmrk}
As observed at the beginning of\/ {\rm Section 2}, for simplicity of notations we have assumed that
$\mu=1$ in problem\/ {\rm (P)}. In the general case,\/ {\rm (\ref{estimate-for-U-and-v})} reads as
$$I(u_\eps)+I(v_\eps)\leqslant c_0 +\frac{S^{N/2}}{N} +  c_1\; \eps^{\frac{N-2}{2}}-
\mu\; c_2\;\eps^{N(1-\frac{q}{p})}.
$$
Thus one still has\/ {\rm (\ref{final-energy-estimate-(H)0})} in case $q=(N+2)/(N-2)$ provided $\mu$ is sufficiently
large.
\end{rmrk}


\begin{thebibliography}{xx}

\bibitem{BTW} A.K.\ Ben-Naoum, C.\ Troestler and M.\ Willem,
   Extrema problems with critical Sobolev exponents on unbounded domains. 
  {\it Nonlinear Anal. TMA} {\bf 26} {(1996)} 823--833.

\bibitem{BCS} G.\ Bianchi, J.\ Chabrowski and A.\ Szulkin,
  On symmetric solutions of an elliptic equation with a nonlinearity involving critical Sobolev
exponent. {\it Nonlinear Anal.\ TMA} {\bf 25} (1995) 41--59.

\bibitem{BN} H.\ Brezis and L.\ Nirenberg,
 Positive solutions of nonlinear elliptic equations involving critical Sobolev exponents.
{\it Comm.\ Pure Appl.\ Math.}\ {\bf 36} (1983) 437--476.

\bibitem{CSS} G.\ Cerami, S.\ Solimini and M.\ Struwe,
 Some existence results for superlinear elliptic boundary value problems involving critical exponents.
{\it J.\ Funct.\ Anal.}\ {\bf 69} (1986) 289--306.

\bibitem{DF} M.\ Del Pino and P.\ Felmer,
Least energy solutions for elliptic equations in unbounded domains.
{\it Proc.\ Royal Soc.\ Edinburgh} {\bf 126A} (1996) 195--208.

\bibitem{GT} D.\ Gilbarg and N.S.\ Trudinger, {\it Elliptic partial 
differential equations of second order,}\/ Second edition.
Grundlehren der Mathematischen Wissenschaften Vol.\ 224,
Springer, New York (1983).

\bibitem{Li} P.L.\ Lions,
The concentration-compactness principle in the Calculus of Variations. The limit case,
part 2. {\it Revista Matem\'atica Iberoamericana} {\bf 1} (1985) 45--121.

\bibitem{RWW} M.\ Ramos, Z.-Q.\ Wang and M.\ Willem,
Positive solutions for elliptic equations with critical growth in unbounded
domains, in {\it Calculus of Variations and Differential Equations,}\/  
A.\ Ioffe, S.\ Reich and I.\ Shafrir Eds.,
Research Notes in Mathematics Series Vol.\ 140,
Chapman \& Hall/CRC, Boca Raton, FL (2000) 192--199.

\bibitem{ST}  I.\ Schindler and K.\ Tintarev, Abstract concentration compactness and elliptic
equations on unbounded domains, in {\it Prog.\ Nonlinear Differential Equations Appl.}\/\ Vol.\ 43, 
M.R.\ Grossinho, M.\ Ramos, C.\ Rebelo and L.\ Sanchez Eds.,
Birkh\"auser, Boston (2001) 369--380.

\bibitem{Ta} G.\ Tarantello,
Nodal solutions of semilinear elliptic equations with critical exponent. 
{\it Differential and Integral Equations} {\bf 5} (1992) 25--42.

\bibitem{Wi} M.\ Willem,
{\it Minimax theorems.}\/
Prog.\ Nonlinear Differential Equations Appl.\ Vol.\ 24,
Birkh\"auser, Boston (1996).

\bibitem{Zhu} X.-P.\ Zhu, 
Multiple entire solutions of a semilinear elliptic equations. 
{\it Nonlinear Analysis TMA} {\bf 12} (1998) 1297--1316.

\end{thebibliography}
\end{document}